\documentclass[12pt]{amsart}

\usepackage{amsmath, amssymb, amscd, newlfont, bbm}
\usepackage[foot]{amsaddr}
\usepackage[final]{hyperref}
\hypersetup{
	colorlinks=true,        
	linkcolor={black},  
	citecolor={black},
	urlcolor={black}     
}
\usepackage{enumitem}

\newcommand{\Met}{{\mathrm{SL}_2(\mathbb R)^{\sim}}}

\newcommand{\spacedcdot}{{\,\cdot\,}}

\newcommand{\bbC}{{\mathbb{C}}}

\newcommand{\bbR}{{\mathbb{R}}}
\newcommand{\bbZ}{{\mathbb{Z}}}

\newcommand{\calH}{{\mathcal{H}}}

\def\babs#1{\pmb{\left|\vphantom{#1}\right.} #1 \pmb{\left.\vphantom{#1}\right|}}

\usepackage[mathscr]{eucal}

\def\SL#1{{\mathrm{SL}}_{#1}}

\providecommand{\abs}[1]{\left\lvert#1\right\rvert}

\providecommand{\scal}[2]{\left<#1,#2\right>}

\numberwithin{equation}{section}

\newtheorem{lemma}[equation]{Lemma}

\newtheorem{theorem}[equation] {Theorem}
\newtheorem{corollary}[equation]{Corollary}

\title
   [ Non-vanishing of $ L $-functions]
   { On the non-vanishing of $ L $-functions associated to cusp forms of half-integral weight}
\author{Sonja \v Zunar}

\address{ Department of Mathematics,
Faculty of Science,
University of Zagreb,
Bijeni\v cka 30,
10000 Zagreb,
Croatia}
\email{szunar@math.hr}

\subjclass[2010]{11F37}
\keywords{$ L $-functions, cusp forms of half-integral weight, non-vanishing of Poincar\' e series, metaplectic cover of $ \mathrm{SL}_2(\mathbb R) $}
\thanks{The author acknowledges Croatian Science Foundation Grant No.~9364.}

\begin{document}
\maketitle

\begin{abstract}
	We prove a strengthening of Mui\' c's integral non-vanishing criterion for Poincar\'e series on unimodular locally compact Hausdorff groups and use it to prove a result on non-vanishing of $ L $-functions associated to cusp forms of half-integral weight.
\end{abstract} 

\section{Introduction}

In \cite[Theorem 4-1]{muic2009construction}, G.~Mui\' c proved an integral criterion for non-vanishing of Poincar\'e series on unimodular locally compact Hausdorff groups. In \cite{MuicIJNT}, this criterion was refined (see \cite[Lemmas 2-1 and 3-1]{MuicIJNT}) and used to study the non-vanishing of classical Poincar\'e series. It found further applications in \cite{MuicJNT} (resp.,~in \cite{ZunarNonVan2017} and \cite{ZunarNonVan2018}), where it was used to study the non-vanishing of Poincar\'e series associated to certain $ K $-finite matrix coefficients of integrable representations of $ \SL2(\bbR) $ (resp., the metaplectic cover of $ \SL2(\bbR) $) and to prove results on non-vanishing of the corresponding cusp forms of integral (resp., of half-integral) weight. In \cite{MuicLFunk}, the criterion provided results on non-vanishing of $ L $-functions associated to cusp forms of integral weight. In the first part of this paper, we prove a strengthening of this criterion (Theorem \ref{thm:l034}) and derive from it a non-vanishing criterion for Poincar\'e series of half-integral weight on the upper half-plane $ \calH $ (Theorem \ref{thm:f098}). Our proofs are considerably shorter than the proofs of criteria in \cite{muic2009construction} and \cite{MuicIJNT} because we eliminate the need to approximate $ L^1 $-functions by continuous functions with compact support.

In the second part of the paper, we use techniques of \cite{MuicLFunk} and Theorem \ref{thm:f098} to prove results on analytic continuation and non-vanishing of $ L $-functions associated to cusp forms of half-integral weight. We work in the following setting. The group $ \SL2(\bbR) $ acts on $ \calH $ by linear fractional transformations:
\[ g.z:=\frac{az+b}{cz+d},\qquad g=\begin{pmatrix}a&b\\c&d\end{pmatrix}\in\SL2(\bbR),\ z\in\calH. \]
The metaplectic cover of $ \SL2(\bbR) $ is defined as the group
\begin{align}
	\Met:=\bigg\{&\sigma=\left(g_\sigma=\begin{pmatrix}a_\sigma&b_\sigma \\c_\sigma&d_\sigma\end{pmatrix},\eta_\sigma\right)\in\SL2(\bbR)\times\bbC^\calH:\nonumber\\
	&\eta_\sigma\text{ is holomorphic and }\eta_\sigma^2(z)=c_\sigma z+d_\sigma\text{ for all }z\in\calH\bigg\}\label{eq:t022}
\end{align}
with multiplication rule  
\begin{equation}\label{eq:t023}
	\sigma_1\sigma_2:=\left(g_{\sigma_1}g_{\sigma_2},\eta_{\sigma_1}(g_{\sigma_2}.z)\eta_{\sigma_2}(z)\right),\qquad \sigma_1,\sigma_2\in\Met.
\end{equation}
Let $ \Gamma $ be a discrete subgroup of finite covolume in $ \Met $ such that $ \infty $ is a cusp of $ P(\Gamma) $, where $ P:\Met\to\SL2(\bbR) $ is the projection onto the first coordinate. Let $ m\in\frac52+\bbZ_{\geq0} $, and let $ \chi:\Gamma\to\bbC^\times $ be a character of finite order such that $ \eta_\gamma^{-2m}=\chi(\gamma) $ for all $ \gamma\in\Gamma_\infty $. The standardly defined space $ S_m(\Gamma,\chi) $ of cusp forms (see Section \ref{sec:002}) is a finite-dimensional Hilbert space under the Petersson inner product $ \scal\spacedcdot\spacedcdot_{S_m(\Gamma,\chi)} $. Let us mention that the spaces $ S_m(N,\psi) $ from \cite{shimura}, where $ N\in4\bbZ_{>0} $ and $ \psi $ is a Dirichlet character modulo $ N $, are of this form (see \cite[\S9]{ZunarNonVan2018}).

For $ f\in S_m(\Gamma,\chi) $ with Fourier expansion $ f(z)=\sum_{n=1}^\infty a_n(f)e^{2\pi in\frac zh} $, the $ L $-function of $ f $ is defined by the formula
\[ L(s,f):=\sum_{n=1}^\infty\frac{a_n(f)}{n^s},\qquad \Re(s)>\frac m2+1. \]
In Section \ref{sec:f146}, we determine Poincar\'e series $ \Psi_{\Gamma,m,\chi,s}\in S_m(\Gamma,\chi) $ and a constant $ c_{\Gamma,m}\in\mathbb R_{>0} $ that depends only on $ \Gamma $ and $ m $, with the following properties: for $ \Re(s)>\frac m2+1 $, we have
\begin{equation}\label{eq:t018}
	L(s,f)=c_{\Gamma,m}\scal{f}{\Psi_{\Gamma,m,\chi,m-\overline s}}_{S_m(\Gamma,\chi)},\qquad f\in S_m(\Gamma,\chi), 
\end{equation}
and if $ m\in\frac92+\bbZ_{\geq0} $, then the formula \eqref{eq:t018} defines an 
analytic continuation of $ L(\spacedcdot,f) $ to the half-plane $ \Re(s)>\frac m2 $. This result is a half-integral weight version of \cite[Theorem 3-9 and (3-10)]{MuicLFunk}.

Next, we apply our non-vanishing criterion, Theorem \ref{thm:f098}, to prove a result on non-vanishing of Poincar\'e series $ \Psi_{\Gamma,m,\chi,m-\overline s} $ for $ \frac m2<\Re(s)<m-1 $. The main idea of the proof is the same as that of the proof of \cite[Lemma 4-2]{MuicLFunk}, but our proof uses stronger estimates which relax the final sufficient condition for the non-vanishing of cusp forms $ \Psi_{\Gamma,m,\chi,m-\overline s} $. The final result is given in Theorem \ref{thm:f121}.

Theorem \ref{thm:f121} translates via \eqref{eq:t018} into a sufficient condition for the inequality  
\[ L(s,\Psi_{\Gamma,m,\chi,m-\overline s})>0   \]
to hold (see Corollaries \ref{cor:l054} and \ref{cor:l055}). This result adds to a series of results on non-vanishing of $ L $-functions associated to certain cusp forms that was started by Kohnen's paper \cite{kohnen1997non}. In that paper, W.~Kohnen studied the non-vanishing in the critical strip of the sum of appropriately normalized Hecke $ L $-functions for $ S_m(\SL2(\bbZ)) $, where $ m\in\bbZ_{\geq4} $. His method is based on estimating the first Fourier coefficient of the cusp form that represents, in the sense of the Riesz representation theorem, the linear functional $ S_m(\SL2(\bbZ))\to\bbC $, $ f\mapsto c_{m,s}L(s,f) $, where $ c_{m,s}\in\bbC^\times $ depends only on $ m $ and $ s $. Kohnen's method inspired Mui\'c's work in \cite{MuicLFunk} and was adapted in \cite{raghuram2005non} to the case of cusp forms with non-trivial level and character, and in \cite{ramakrishnan2014non} and \cite{kohnen2017non} to the case of cusp forms of half-integral weight. \vspace{5mm}\\
This paper grew out of my PhD thesis. I would like to thank my advisor, Goran Mui\'c, for his encouragement, support, and many discussions. I would also like to thank Petar Baki\'c for some useful comments.

\section{Preliminaries}\label{sec:002}

Let $ \sqrt\spacedcdot:\bbC\to\bbC $ be the branch of the complex square root such that $ \arg\left(\sqrt z\right)\in\left]-\frac\pi2,\frac\pi2\right] $ for all $ z\in\bbC^\times $. We write $ i:=\sqrt{-1} $. 

The group $ \SL2(\bbR) $ acts on $ \bbC\cup\left\{\infty\right\} $ by linear fractional transformations:
\[ g.z:=\frac{az+b}{cz+d},\qquad g=\begin{pmatrix}a&b\\c&d\end{pmatrix}\in\SL2(\bbR),\ z\in\bbC\cup\left\{\infty\right\}. \]
The half-plane $ \calH:=\bbC_{\Im(z)>0} $ is an orbit of this action.
We note that
\begin{equation}\label{eq:t011}
	\Im(g.z)=\frac{\Im(z)}{\abs{cz+d}^2},\qquad g=\begin{pmatrix}a&b\\c&d\end{pmatrix}\in\SL2(\bbR),\ z\in\calH.
\end{equation}

We define the group $ \Met $ by formulae \eqref{eq:t022} and \eqref{eq:t023}.
We use shorthand notation $ \left(g_\sigma,\eta_\sigma(i)\right) $ for elements $ \sigma=\left(g_\sigma,\eta_\sigma\right) $ of $ \Met $ and define the smooth structure of $ \Met $ by requiring that the Iwasawa parametrization
$ \bbR\times\bbR_{>0}\times\bbR\to\Met $,
\begin{equation}\label{eq:001}
	(x,y,t)\mapsto\left(\begin{pmatrix}1&x\\0&1\end{pmatrix},1\right)\left(\begin{pmatrix}y^{\frac12}&0\\0&y^{-\frac12}\end{pmatrix},y^{-\frac14}\right)\left(\begin{pmatrix}\cos t&-\sin t\\\sin t&\cos t\end{pmatrix},e^{i\frac t2}\right),
\end{equation} 
be a local diffeomorphism. With this smooth structure, $ \Met $ is a connected Lie group, and the projection $ P:\Met\to\SL2(\bbR) $, $ P(\sigma):=g_\sigma $, is a smooth covering homomorphism of degree $ 2 $.
Next, let us denote the center of $ \Met $ by $ Z\left(\Met\right) $. For a discrete subgroup $ \Gamma $ of $ \Met $, we define
\[ Z(\Gamma):=\Gamma\cap Z\left(\Met\right)\qquad\text{and}\qquad \varepsilon_\Gamma:=\abs{Z(\Gamma)}. \]

The group $ \Met $ acts on $ \calH $ by the formula
\[ \sigma.z:=g_\sigma.z,\qquad\sigma\in\Met,\ z\in\calH.  \]
Let us denote the three factors on the right-hand side of \eqref{eq:001}, from left to right, by $ n_x $, $ a_y $, and $ \kappa_t $. We have
\begin{equation}\label{eq:t019}
	n_xa_y\kappa_t.i=x+iy,\qquad x\in\bbR,\ y\in\bbR_{>0},\ t\in\bbR. 
\end{equation}
The group $ K:=\left\{\sigma\in\Met:\sigma.i=i\right\}=\left\{\kappa_t:t\in\bbR\right\} $ is a maximal compact subgroup of $ \Met $. Its unitary dual consists of the characters $ \chi_n:K\to\bbC^\times $, $ n\in\frac12\bbZ $,
\[ \chi_n(\kappa_t):=e^{-int},\qquad t\in\bbR. \]

Let $ v $ be the $ \SL2(\bbR) $-invariant Radon measure on $ \calH $ defined by the formula 
\[ \int_\calH f\,dv=\int_\bbR\int_0^\infty f(x+iy)y^{-2}\,dy\,dx,\qquad f\in C_c(\calH). \]
We fix the following Haar measure on $ \Met $: 
\[ \int_{\Met}\varphi\,d\mu_{\Met}:=\frac1{4\pi}\int_0^{4\pi}\int_\calH\varphi(n_xa_y\kappa_t)\,dv(x+iy)\,dt \]
for all $ \varphi\in C_c\left(\Met\right) $.

Next, let $ G $ be a locally compact Hausdorff group that is second-countable and unimodular, with Haar measure $ \mu_G $. For a discrete subgroup $ \Gamma $ of $ G $, we denote by $ \mu_{\Gamma\backslash G} $ the unique Radon measure on $ \Gamma\backslash G $ such that
\[ \int_{\Gamma\backslash G}\sum_{\gamma\in\Gamma}\varphi(\gamma g)\,d\mu_{\Gamma\backslash G}(g)=\int_G\varphi\,d\mu_G,\qquad\varphi\in C_c(G). \]
We define the space $ L^1(\Gamma\backslash G) $ with respect to this measure. We note that if $ \Lambda $ is a subgroup of $ \Gamma $, then
\begin{equation}\label{eq:t001}
	\int_{\Gamma\backslash G}\sum_{\gamma\in\Lambda\backslash\Gamma}\varphi(\gamma g)\,d\mu_{\Gamma\backslash G}(g)=\int_{\Lambda\backslash G}\varphi\,d\mu_{\Lambda\backslash G},\qquad\varphi\in L^1(\Lambda\backslash G)
\end{equation}
(see \cite[(2-2)]{MuicIJNT}). In the case when $ G=\Met $ and $ \mu_G=\mu_\Met $, we have, for all $ \varphi\in C_c\left(\Gamma\backslash\Met\right) $,
\begin{equation}\label{eq:t009}
	\int_{\Gamma\backslash\Met}\varphi\,d\mu_{\Gamma\backslash\Met}=\frac1{4\pi\varepsilon_\Gamma}\int_0^{4\pi}\int_{\Gamma\backslash\calH}\varphi(n_xa_y\kappa_t)\,dv(x+iy)\,dt.
\end{equation}

For every $ m\in\frac12+\bbZ_{\geq0} $, we have the following right action of $ \Met $ on $ \bbC^\calH $:
\[ \left(f\big|_m\sigma\right)(z):=f(g_\sigma.z)\eta_\sigma(z)^{-2m},\qquad f\in\bbC^\calH,\ \sigma\in\Met,\ z\in\calH. \]
Let us note that for $ \delta\in Z\left(\Met\right)=P^{-1}(\left\{\pm I_2\right\})=\left\{\kappa_{n\pi}:n\in\bbZ\right\} $, we have
\begin{equation}\label{eq:t014}
	\eta_\delta^{-2m}(z)=\chi_m(\delta),\qquad z\in\calH,
\end{equation}
and
\begin{equation}\label{eq:t008}
	f\big|_m\delta=\chi_m(\delta)f,\qquad f\in\bbC^\calH.
\end{equation}

Next, let $ \Gamma $ be a discrete subgroup of finite covolume in $ \Met $, $ \chi:\Gamma\to\bbC^\times $ a character of finite order, and $ m\in\frac12+\bbZ_{\geq0} $. The space $ S_m(\Gamma,\chi) $ of cusp forms for $ \Gamma $ of weight $ m $ with character $ \chi $ by definition consists of holomorphic functions $ f:\calH\to\bbC $ that satisfy 
\begin{equation}\label{eq:t012}
	f\big|_m\gamma=\chi(\gamma)f,\qquad\gamma\in\Gamma,
\end{equation}
and vanish at all cusps of $ P(\Gamma) $ (see the beginning of \cite[\S3]{ZunarNonVan2018} for a detailed explanation of the last condition). $ S_m(\Gamma,\chi) $ is a finite-dimensional Hilbert space under the Petersson inner product
\[ \scal{f_1}{f_2}_{S_m(\Gamma,\chi)}:=\varepsilon_\Gamma^{-1}\int_{\Gamma\backslash\calH}f_1(z)\overline{f_2(z)}\Im(z)^m\,dv(z),\qquad f_1,f_2\in S_m(\Gamma,\chi). \]

For $ m\in\frac12+\bbZ_{\geq0} $, we introduce the classical lift of a function $ f:\calH\to\bbC $ to the function $ F_f:\Met\to\bbC $,
\[ F_f(\sigma):=\left(f\big|_m\sigma\right)(i). \]
Equivalently,
\begin{equation}\label{eq:t010}
	F_f(n_xa_y\kappa_t)=f(x+iy)y^{\frac m2}e^{-imt},\qquad x\in\bbR,\ y\in\bbR_{>0},\ t\in\bbR. 
\end{equation}
Let $ \Gamma $ be a discrete subgroup of $ \Met $ and $ \chi:\Gamma\to\bbC^\times $ a unitary character. One checks easily that for every $ f:\calH\to\bbC $, the following equivalence holds:
\begin{equation}\label{eq:t005}
	f\big|_m\gamma=\chi(\gamma)f,\quad \gamma\in\Gamma\qquad\Leftrightarrow\qquad F_f(\gamma\spacedcdot)=\chi(\gamma)F_f,\quad \gamma\in\Gamma.
\end{equation}
Moreover, if these equivalent conditions are satisfied and $ f $ is measurable, then we have 
\begin{equation}\label{eq:t006}
	\int_{\Gamma\backslash\Met}\abs{F_f}\,d\mu_{\Gamma\backslash\Met}\underset{\eqref{eq:t010}}{\overset{\eqref{eq:t009}}=}\varepsilon_\Gamma^{-1}\int_{\Gamma\backslash\calH}\abs{f(z)\Im(z)^{\frac m2}}\,dv(z).
\end{equation}

\section{A non-vanishing criterion for Poincar\'e series on locally compact Hausdorff groups}\label{subsec:f084}

The proof of the following lemma is given in the first part of \cite[proof of Lemma 2-1]{MuicIJNT}.

\begin{lemma}\label{lem:t005}
	Let $ G $ be a locally compact Hausdorff group that is second-countable and unimodular, with Haar measure $ \mu_G $. Let $ \Gamma $ be a discrete subgroup of $ G $, $ \Lambda $ a subgroup of $ \Gamma $, and $ \chi:\Gamma\to\bbC^\times $ a unitary character. Let $ \varphi:G\to\bbC $ be a measurable function with the following properties:
	\begin{enumerate}[label=\textup{(F\arabic*)},leftmargin=*,align=left]
		\item[{\textup{(F1)}}] $ \varphi(\lambda g)=\chi(\lambda)\varphi(g),\quad\lambda\in\Lambda,\ g\in G $
		\item[{\textup{(F2)}}]	$ \abs\varphi\in L^1(\Lambda\backslash G) $.	
	\end{enumerate}
	Then, we have the following:
	\begin{enumerate}[label=\textup{(\arabic*)},leftmargin=*,align=left]
		\item\label{lem:t005:1} The Poincar\'e series
		\[ \left(P_{\Lambda\backslash\Gamma,\chi}\varphi\right)(g):=\sum_{\gamma\in\Lambda\backslash\Gamma}\overline{\chi(\gamma)}\varphi(\gamma g),\qquad g\in G, \]
		converges absolutely almost everywhere on $ G $.
		\item\label{lem:t005:2} $ \left(P_{\Lambda\backslash\Gamma,\chi}\varphi\right)(\gamma \spacedcdot)=\chi(\gamma)P_{\Lambda\backslash\Gamma,\chi}\varphi,\quad\gamma\in\Gamma.  $
		\item\label{lem:t005:3} $ \abs{P_{\Lambda\backslash\Gamma,\chi}\varphi}\in L^1(\Gamma\backslash G) $.
	\end{enumerate}
\end{lemma}

The following theorem is a strengthening of the integral non-vanishing criterion \cite[Lemma 2-1]{MuicIJNT}.

\begin{theorem}\label{thm:l034}
	Let $ G $, $ \Gamma $, $ \Lambda $, $ \chi $, and $ \varphi $ satisfy the assumptions of Lemma \ref{lem:t005}. Then, 
	\[ \int_{\Gamma\backslash G}\abs{P_{\Lambda\backslash\Gamma,\chi}\varphi}\,d\mu_{\Gamma\backslash G}>0  \]
	if there exists a Borel measurable set $ C\subseteq G $ with the following properties:
	\begin{enumerate}[label=\textup{(C\arabic*)},leftmargin=*,align=left]
		\item\label{enum:l034:1} $ CC^{-1}\cap\Gamma\subseteq\Lambda $.
		\item\label{enum:l034:2} Writing $ (\Lambda C)^c:=G\setminus\Lambda C $, we have
		\begin{equation}\label{eq:l040}
			\int_{\Lambda\backslash\Lambda C}\babs \varphi\,d\mu_{\Lambda\backslash G}>\int_{\Lambda\backslash(\Lambda C)^c}\babs \varphi\,d\mu_{\Lambda\backslash G}
		\end{equation}
		for some measurable function $ \babs\spacedcdot:\bbC\to\bbR_{\geq0} $ with the following properties:
		\begin{enumerate}[label=\textup{(M\arabic*)},leftmargin=*,align=left]
			\item\label{enum:l036:1} $ \babs0=0 $.
			\item\label{enum:l036:2} $ \babs z=\babs{\abs z},\quad z\in \bbC $.
			\item\label{enum:l036:3} $ \babs{\sum_{n=1}^\infty z_n}\leq\sum_{n=1}^\infty\babs{z_n} $ for every $ (z_n)_{n\in\bbZ_{>0}}\subseteq\bbC $ such that $ \sum_{n=1}^\infty\abs{z_n}<\infty $.
		\end{enumerate} 
	\end{enumerate}
\end{theorem}

\begin{proof}
	Suppose that a Borel measurable set $ C\subseteq G $ satisfies (C1) and (C2). Let us denote by $ \mathbbm1_A $ the characteristic function of $ A\subseteq G $. The property (C1) implies the following:
	\begin{equation}\label{eq:l035}
		\text{for every }g\in G,\ \#\left\{\gamma\in\Lambda\backslash\Gamma:\mathbbm1_{\Lambda C}(\gamma g)\neq0\right\}\leq1. 
	\end{equation}
	Namely, if $ \mathbbm1_{\Lambda C}(\gamma g)\neq0 $ and $ \mathbbm1_{\Lambda C}\left(\gamma' g\right)\neq0 $ for some $ \gamma,\gamma'\in\Gamma $, i.e., if $ \gamma g,\gamma'g\in\Lambda C $, then $ \gamma\gamma'^{-1}=(\gamma g)\left(\gamma'g\right)^{-1}\in\Lambda CC^{-1}\Lambda\cap\Gamma=\Lambda\left(CC^{-1}\cap\Gamma\right)\Lambda\overset{\text{(C1)}}=\Lambda $, hence $ \Lambda\gamma=\Lambda\gamma' $.
	
	Now we have
	\begin{align}
		&\int_{\Gamma\backslash G}\babs{P_{\Lambda\backslash\Gamma,\chi}\left(\varphi\cdot\mathbbm1_{\Lambda C}\right)}\,d\mu_{\Gamma\backslash G}\nonumber\\
		&\underset{\phantom{\text{(M1),(M2)}}}=\int_{\Gamma\backslash G}\babs{\sum_{\gamma\in\Lambda\backslash\Gamma}\overline{\chi(\gamma)}\varphi(\gamma g)\mathbbm1_{\Lambda C}(\gamma g)}\,d\mu_{\Gamma\backslash G}(g)\nonumber\\
		&\underset{\text{(M1),(M2)}}{\overset{\eqref{eq:l035}}=}\int_{\Gamma\backslash G}\sum_{\gamma\in\Lambda\backslash\Gamma}\babs{\varphi(\gamma g)\mathbbm1_{\Lambda C}(\gamma g)}\,d\mu_{\Gamma\backslash G}(g)\nonumber\\
		&\underset{\phantom{\text{(M1),(M2)}}}{\overset{\eqref{eq:t001}}=}\int_{\Lambda\backslash G}\babs{\varphi\cdot\mathbbm1_{\Lambda C}}\,d\mu_{\Lambda\backslash G}\nonumber\\
		&\underset{\phantom{\text{(M1),(M2)}}}{\overset{\text{(M1)}}=}\int_{\Lambda\backslash\Lambda C}\babs \varphi\,d\mu_{\Lambda\backslash G}.\label{eq:l038}
	\end{align}
	
	On the other hand,
	\begin{align}
		&\int_{\Gamma\backslash G}\babs{P_{\Lambda\backslash\Gamma,\chi}\left(\varphi\cdot\mathbbm1_{\left(\Lambda C\right)^c}\right)}\,d\mu_{\Gamma\backslash G}\nonumber\\
		&\underset{\phantom{\text{(M2)}}}=\int_{\Gamma\backslash G}\babs{\sum_{\gamma\in\Lambda\backslash\Gamma}\overline{\chi(\gamma)}\varphi(\gamma g)\mathbbm1_{\left(\Lambda C\right)^c}(\gamma g)}\,d\mu_{\Gamma\backslash G}(g)\nonumber\\
		&\underset{\text{(M2)}}{\overset{\text{(M3)}}\leq}\int_{\Gamma\backslash G}\sum_{\gamma\in\Lambda\backslash\Gamma}\babs{\varphi(\gamma g)\mathbbm1_{\left(\Lambda C\right)^c}(\gamma g)}\,d\mu_{\Gamma\backslash G}(g)\nonumber\\	
		&\underset{\phantom{\text{(M2)}}}{\overset{\eqref{eq:t001}}=}\int_{\Lambda\backslash G}\babs{\varphi\cdot\mathbbm1_{\left(\Lambda C\right)^c}}\,d\mu_{\Lambda\backslash G}\nonumber\\
		&\overset{\text{(M1)}}=\int_{\Lambda\backslash\left(\Lambda C\right)^c}\babs \varphi\,d\mu_{\Lambda\backslash G}.\label{eq:l039}
	\end{align}
	
	We also note that
	\begin{equation}\label{eq:l037}
		\babs{z+w}\overset{\text{(M3)}}\geq\babs z-\babs{-w}\overset{\text{(M2)}}=\babs z-\babs w,\qquad z,w\in\bbC.
	\end{equation}
	
	Now we have
	\begin{align*}
		&\int_{\Gamma\backslash G}\babs{P_{\Lambda\backslash\Gamma,\chi}\varphi}\,d\mu_{\Gamma\backslash G}\\
		&\overset{\eqref{eq:l037}}\geq\int_{\Gamma\backslash G}\babs{P_{\Lambda\backslash\Gamma,\chi}\left(\varphi\cdot\mathbbm1_{\Lambda C}\right)}\,d\mu_{\Gamma\backslash G}
		-\int_{\Gamma\backslash G}\babs{P_{\Lambda\backslash\Gamma,\chi}\left(\varphi\cdot\mathbbm1_{(\Lambda C)^c}\right)}\,d\mu_{\Gamma\backslash G}\\
		&\overset{\eqref{eq:l038}}{\underset{\eqref{eq:l039}}\geq}\int_{\Lambda\backslash\Lambda C}\babs{\varphi}\,d\mu_{\Lambda\backslash G}-\int_{\Lambda\backslash(\Lambda C)^c}\babs{\varphi}\,d\mu_{\Lambda\backslash G}\\
		&\overset{\eqref{eq:l040}}>0,
	\end{align*}
	from which it follows by (M1) that $ \abs{P_{\Lambda\backslash\Gamma,\chi}\varphi}\not\equiv0 $ in $ L^1(\Gamma\backslash G) $, hence 
	\[ \int_{\Gamma\backslash G}\abs{P_{\Lambda\backslash\Gamma,\chi}\varphi}\,d\mu_{\Gamma\backslash G}>0.\qedhere \]
\end{proof}

In the case when $ \babs\spacedcdot=\abs\spacedcdot:\bbC\to\bbR_{\geq0} $, Theorem \ref{thm:l034} is a variant of \cite[Lemma 2-1]{MuicIJNT}. More generally, we have the following lemma.

\begin{lemma}
	Let $ f:\bbR_{\geq0}\to\bbR_{\geq0} $ such that $ f(0)=0 $. Suppose that $ f $ is concave, i.e., that
	\[ f\big((1-t)x+ty\big)\geq(1-t)f(x)+tf(y),\qquad t\in[0,1],\ x,y\in\bbR_{\geq0}. \]
	Then, the function $ \babs\spacedcdot:\bbC\to\bbR_{\geq0} $,
	\[ \babs z:=f\left(\abs{z}\right), \]
	is measurable and has the properties \textup{(M1)}--\textup{(M3)} of Theorem \ref{thm:l034}.
\end{lemma}

\begin{proof}
	We note that $ f $ is subadditive and non-decreasing by \cite[Theorem 2]{zhi2016concave}. Next, we prove that $ \babs\spacedcdot $ satisfies (M3), the other claims being trivial. 
	
	First, let us show that for every $ (x_n)_{n\in\bbZ_{>0}}\subseteq\bbR_{\geq0} $ such that $ \sum_{n=1}^\infty x_n<\infty $, we have
	\begin{equation}\label{eq:t002}
		f\left(\sum_{n=1}^\infty x_n\right)\leq\sum_{n=1}^\infty f(x_n).
	\end{equation}
	This holds trivially if $ \sum_{n=1}^\infty x_n=0 $, so suppose that $ \sum_{n=1}^\infty x_n>0 $. By subadditivity, we have $ f\left(\sum_{n=1}^N x_n\right)\leq \sum_{n=1}^Nf(x_n) $ for all $ N\in\bbZ_{>0} $, from which  \eqref{eq:t002} follows by applying $ \lim_{N\to\infty} $ since by concavity $ f $ is continuous on $ \bbR_{>0} $ \cite[Theorem 1.3.3]{niculescu2006convex}.
	
	Now, for every $ (z_n)_{n\in\bbZ_{>0}}\subseteq\bbC $ such that $ \sum_{n=1}^\infty\abs{z_n}<\infty $, we have
	\[ \babs{\sum_{n=1}^\infty z_n}=f\left(\abs{\sum_{n=1}^\infty z_n}\right)\leq f\left(\sum_{n=1}^\infty\abs{z_n}\right)\overset{\eqref{eq:t002}}\leq\sum_{n=1}^\infty f\left(\abs{z_n}\right)=\sum_{n=1}^\infty\babs{z_n}, \]
	where the first inequality holds because $ f $ is non-decreasing.
\end{proof}

\section{A non-vanishing criterion for Poincar\'e series of half-integral weight}

\begin{lemma}\label{lem:t007}
	Let $ \Gamma $ be a discrete subgroup of $ \Met $, $ \Lambda $ a subgroup of $ \Gamma $, $ \chi:\Gamma\to\bbC^\times $ a unitary character, and $ m\in\frac12+\bbZ_{\geq0} $. Let $ f:\calH\to\bbC $ be a measurable function with the following properties:
	\begin{enumerate}[label=\textup{(f\arabic*)},leftmargin=*,align=left]
		\item $ f\big|_m\lambda=\chi(\lambda)f $,\quad$ \lambda\in\Lambda $.
		\item $ \int_{\Lambda\backslash\calH}\abs{f(z)\Im(z)^{\frac m2}}\,dv(z)<\infty $.
	\end{enumerate}
	Then, we have the following:
	\begin{enumerate}[label=\textup{(\arabic*)},leftmargin=*,align=left]
		\item\label{lem:t007:1} The series
		\[ P_{\Lambda\backslash\Gamma,\chi}f:=\sum_{\gamma\in\Lambda\backslash\Gamma}\overline{\chi(\gamma)}f\big|_m\gamma \]
		converges absolutely almost everywhere on $ \calH $.
		\item\label{lem:t007:2} $ \left(P_{\Lambda\backslash\Gamma,\chi}f\right)\big|_m\gamma=\chi(\gamma)P_{\Lambda\backslash\Gamma,\chi}f,\quad\gamma\in\Gamma $.
		\item\label{lem:t007:3} $ \int_{\Gamma\backslash\calH}\abs{\left(P_{\Lambda\backslash\Gamma,\chi}f\right)(z)\Im(z)^{\frac m2}}\,dv(z)=\varepsilon_\Gamma\int_{\Gamma\backslash\Met}\abs{{P_{\Lambda\backslash\Gamma,\chi}F_f}}\,d\mu_{\Gamma\backslash\Met}$ $<\infty $.
	\end{enumerate}
\end{lemma}

\begin{proof}
	By (f1), the terms of the series $ P_{\Lambda\backslash\Gamma,\chi}f $ are well-defined. Next, the function $ F_f $ satisfies the assumptions of Lemma \ref{lem:t005}: it satisfies (F1) by \eqref{eq:t005} and (f1), and it satisfies (F2) by \eqref{eq:t006} and (f2). By Lemma \ref{lem:t005}.\ref{lem:t005:1}, the series $ P_{\Lambda\backslash\Gamma,\chi}F_f $ converges absolutely almost everywhere on $ \Met $, and this implies the claim \ref{lem:t007:1} and the equality
	\begin{equation}\label{eq:t020}
		F_{P_{\Lambda\backslash\Gamma,\chi}f}=P_{\Lambda\backslash\Gamma,\chi}F_f,
	\end{equation}
	as is easily checked by following the definitions. The claim \ref{lem:t007:2} is easy to check directly; alternatively, it follows by \eqref{eq:t005} and \eqref{eq:t020} from Lemma \ref{lem:t005}.\ref{lem:t005:2}. Finally, the equality in the claim \ref{lem:t007:3} holds by \eqref{eq:t006} and \eqref{eq:t020}, 
	and its right-hand side is finite by Lemma \ref{lem:t005}.\ref{lem:t005:3}.
\end{proof}

\begin{lemma}\label{lem:f088}
	Let $ \Gamma $, $ \Lambda $, $ \chi $, $ m $, and $ f $ satisfy the assumptions of Lemma \ref{lem:t007}. If
	\begin{equation}\label{eq:f080}
		\chi\big|_{Z(\Gamma)}\neq\chi_m\big|_{Z(\Gamma)},
	\end{equation}
	then $ P_{\Lambda\backslash\Gamma,\chi}f\equiv0 $.
\end{lemma}

\begin{proof}
	The claim is clear from the equality
	\[ \chi_m(\delta)P_{\Lambda\backslash\Gamma,\chi}f\overset{\eqref{eq:t008}}=\left(P_{\Lambda\backslash\Gamma,\chi}f\right)\big|_m\delta=\chi(\delta)P_{\Lambda\backslash\Gamma,\chi}f,\qquad\delta\in Z(\Gamma), \]
	where the second equality holds by Lemma \ref{lem:t007}.\ref{lem:t007:2}.
\end{proof}

\begin{theorem}\label{thm:f098}
	Let $ \Gamma $, $ \Lambda $, $ \chi $, $ m $, and $ f $ satisfy the assumptions of Lemma \ref{lem:t007}. Suppose that	
	\begin{equation}\label{eq:f081}
		\chi\big|_{Z(\Gamma)}=\chi_m\big|_{Z(\Gamma)}.
	\end{equation}
	Then, 
	\begin{equation}\label{eq:f090}
		\int_{\Gamma\backslash\calH}\abs{\left(P_{\Lambda\backslash\Gamma,\chi}f\right)(z)\Im(z)^{\frac m2}}\,dv(z)>0 
	\end{equation}
	if there exists a Borel measurable set $ S\subseteq\calH $ with the following properties:
	\begin{enumerate}[label=\textup{(S\arabic*)},leftmargin=*,align=left]
		\item If $ z_1,z_2\in S $ and $ z_1\neq z_2 $, then $ \Gamma.z_1\neq\Gamma.z_2 $.\label{enum:022:1}
		\item\label{enum:022:2b} Writing $ (\Lambda.S)^c:=\calH\setminus\Lambda.S $, we have
		\[ \int_{\Lambda\backslash\Lambda.S}\babs{f(z)\Im(z)^{\frac m2}}\,dv(z)>\int_{\Lambda\backslash(\Lambda.S)^c}\babs{f(z)\Im(z)^{\frac m2}}\,dv(z) \] 
		for some measurable function $ \babs{\spacedcdot}:\bbC\to\bbR_{\geq0} $ that satisfies \textup{(M1)--(M3)} of Theorem \ref{thm:l034}.
	\end{enumerate}
\end{theorem}

\begin{proof}
	Suppose that a Borel measurable set $ S\subseteq \calH $ satisfies (S1) and (S2). We have
	\begin{align*}
		P_{\Lambda\backslash\Gamma,\chi}f&\underset{\phantom{\eqref{eq:f081}}}=\sum_{\gamma\in\Lambda Z(\Gamma)\backslash\Gamma}\ \sum_{\delta\in\Lambda\backslash\Lambda Z(\Gamma)}\overline{\chi(\delta\gamma)}f\big|_m\delta\gamma\\
		&\underset{\phantom{\eqref{eq:f081}}}{\overset{\eqref{eq:t008}}=}\sum_{\gamma\in\Lambda Z(\Gamma)\backslash\Gamma}\left(\sum_{\delta\in Z(\Lambda)\backslash Z(\Gamma)}\overline{\chi(\delta)}\chi_m(\delta)\right)\overline{\chi(\gamma)}f\big|_m\gamma\\
		&\overset{\eqref{eq:f081}}=\frac{\varepsilon_\Gamma}{\varepsilon_\Lambda}P_{\Lambda Z(\Gamma)\backslash\Gamma,\chi}f,
	\end{align*}
	hence it suffices to prove the non-vanishing of the series $ P_{\Lambda Z(\Gamma)\backslash\Gamma,\chi}f $. In other words, we may assume without loss of generality that $ Z(\Gamma)\subseteq\Lambda $. We may also assume that $ S $ contains no elliptic points for $ \Gamma $. Under these assumptions, by (S1) and \eqref{eq:t019} the set
	\[ C:=\left\{n_xa_y:x+iy\in S\right\}K\subseteq\Met \]
	satisfies $ CC^{-1}\cap\Gamma\subseteq Z(\Gamma)\subseteq\Lambda $, i.e., $ C $ has the property (C1) of Theorem \ref{thm:l034}. 
	
	Our goal is to apply Theorem 3-2 to the series $ P_{\Lambda\backslash\Gamma,\chi}F_f $. Since we have seen in the proof of Lemma \ref{lem:t007} that the function $ F_f $ satisfies the conditions (F1) and (F2) of Lemma \ref{lem:t005}, it remains to prove the inequality
	\begin{equation}\label{eq:t021}
		\int_{\Lambda\backslash\Lambda C}\babs{F_f}\,d\mu_{\Lambda\backslash\Met}>\int_{\Lambda\backslash(\Lambda C)^c}\babs{F_f}\,d\mu_{\Lambda\backslash\Met},
	\end{equation}
	where $ (\Lambda C)^c:=\Met\setminus\Lambda C $.
	To this end, we note that
	\begin{equation}\label{eq:f085}
		\Lambda C=\left\{n_xa_y:x+iy\in\Lambda.S\right\}K,
	\end{equation}
	hence we have
	\begin{align}
		&\int_{\Lambda\backslash\Lambda C}\babs{F_f}\,d\mu_{\Lambda\backslash\Met}\nonumber\\
		&\underset{\phantom{\eqref{eq:t010},\textup{(M2)}}}{\overset{\eqref{eq:t009}}=}\frac1{4\pi\varepsilon_\Lambda}\int_0^{4\pi}\int_{\Lambda \backslash\calH}\babs{F_f(n_xa_y\kappa_t)}\,\mathbbm1_{\Lambda C}(n_xa_y\kappa_t)\,dv(x+iy)\,dt\nonumber\\
		&\overset{\eqref{eq:t010},\textup{(M2)}}{\underset{\eqref{eq:f085}}=}\frac1{4\pi\varepsilon_\Lambda}\int_0^{4\pi}\int_{\Lambda\backslash\calH}\babs{f(x+iy)y^{\frac m2}}\,\mathbbm1_{\Lambda.S}(x+iy)\,dv(x+iy)\,dt\nonumber\\
		&\underset{\phantom{\eqref{eq:t010},\textup{(M2)}}}=\varepsilon_\Lambda^{-1}\int_{\Lambda\backslash\Lambda.S}\babs{f(z)\Im(z)^{\frac m2}}\,dv(z).\label{eq:f086}
	\end{align}
	We see analogously that
	\begin{equation}\label{eq:f087}
		\int_{\Lambda\backslash(\Lambda C)^c}\babs{F_f}\,d\mu_{\Lambda\backslash\Met}=\varepsilon_\Lambda^{-1}\int_{\Lambda\backslash(\Lambda.S)^c}\babs{f(z)\Im(z)^{\frac m2}}\,dv(z).
	\end{equation}
	By \eqref{eq:f086} and \eqref{eq:f087}, (S2) implies \eqref{eq:t021}.
	
	Thus, by Theorem \ref{thm:l034} $ \int_{\Gamma\backslash\Met}\abs{P_{\Lambda\backslash\Gamma,\chi}F_f}\,d\mu_{\Gamma\backslash\Met}>0 $. By Lemma \ref{lem:t007}.\ref{lem:t007:3}, this implies \eqref{eq:f090}.
\end{proof}

We conclude this section by one more lemma on convergence of Poincar\' e series of half-integral weight (cf.~\cite[Lemma 2-4]{MuicLFunk}).

\begin{lemma}\label{lem:t011}
	Let $ \Gamma $, $ \Lambda $, $ \chi $, $ m $, and $ f $ satisfy the assumptions of Lemma \ref{lem:t007}. Suppose in addition that $ f $ is holomorphic. Then, the series $ P_{\Lambda\backslash\Gamma,\chi}f $ converges absolutely and locally uniformly on $ \calH $ and defines an element of $ S_m(\Gamma,\chi) $.
\end{lemma}

\begin{proof}
	The series $ P_{\Lambda\backslash\Gamma,\chi}f $ converges absolutely and locally uniformly on $ \calH $ by the half-integral weight version of \cite[Theorem 2.6.6.(1)]{miyake}. Since Lemma \ref{lem:t007}.\ref{lem:t007:2} also holds, it remains to prove that $ P_{\Lambda\backslash\Gamma,\chi}f $ vanishes at every cusp of $ P(\Gamma) $.
	
	Let $ x $ be a cusp of $ P(\Gamma) $, and let $ \sigma\in\Met $ such that $ \sigma.\infty=x $. Let
	\[ \left(\left(P_{\Lambda\backslash\Gamma,\chi}f\right)\big|_m\sigma\right)(z)=\sum_{n\in\bbZ}a_ne^{2\pi in\frac z{h'}},\qquad z\in\calH, \]
	be the Fourier expansion of $ \left(P_{\Lambda\backslash\Gamma,\chi}f\right)\big|_m\sigma $. We need to prove that $ a_n=0 $ for all $ n\in\bbZ_{\leq0} $. Let $ h\in\bbR_{>0} $ such that  $ Z\left(\Met\right)\sigma^{-1}\Gamma_x\sigma=Z\left(\Met\right)\left<n_h\right> $. Since $ \Big|\left(P_{\Lambda\backslash\Gamma,\chi}f\right)\big|_m\sigma\Big| $ is $ h $-periodic, by the half-integral weight version of \cite[Lemma 2-1]{MuicLFunk} it suffices to prove that
	\[ I:=\int_{[0,h[\times]h,\infty[}\abs{\left(\left(P_{\Lambda\backslash\Gamma,\chi}f\right)\big|_m\sigma\right)(z)\Im(z)^{\frac m2}}\,dv(z) \]
	is finite. We have
	\begin{align*}
		I&\overset{\eqref{eq:t011}}=\int_{[0,h[\times]h,\infty[}\abs{\left(P_{\Lambda\backslash\Gamma,\chi}f\right)(\sigma.z)\Im(\sigma.z)^{\frac m2}}\,dv(z)\\
		&\overset{\phantom{\eqref{eq:t011}}}=\int_{\sigma.\left([0,h[\times]h,\infty[\right)}\abs{\left(P_{\Lambda\backslash\Gamma,\chi}f\right)(z)\Im(z)^{\frac m2}}\,dv(z)\\
		&\overset{\phantom{\eqref{eq:t011}}}\leq\int_{\Gamma\backslash\calH}\sum_{\gamma\in\Lambda\backslash\Gamma}\abs{\left(f\big|_m\gamma\right)(z)\Im(z)^{\frac m2}}\,dv(z)\\
		&\overset{\eqref{eq:t011}}=\int_{\Gamma\backslash\calH}\sum_{\gamma\in\Lambda\backslash\Gamma}\abs{f(\gamma.z)\Im(\gamma.z)^{\frac m2}}\,dv(z)\\
		&\overset{\phantom{\eqref{eq:t011}}}=\frac{\varepsilon_\Gamma}{\varepsilon_\Lambda}\int_{\Lambda\backslash\calH}\abs{f(z)\Im(z)^{\frac m2}}\,dv(z)\overset{\textrm{(f2)}}<\infty,
	\end{align*}
	where the first inequality holds because by \cite[Corollary 1.7.5]{miyake} any two distinct points of $ \sigma.\left([0,h[\times]h,\infty[\right) $ are mutually $ \Gamma $-inequivalent.
\end{proof}

\section{Analytic continuation of $ L $-functions}\label{sec:f146}

Throughout this section, let $ m\in\frac52+\bbZ_{\geq0} $, let $ \Gamma $ be a discrete subgroup of finite covolume in $ \Met $ such that $ \infty $ is a cusp of $ P(\Gamma) $, and let $ \chi:\Gamma\to\bbC^\times $ be a character of finite order such that
\begin{equation}\label{eq:l051}
	\eta_\gamma^{-2m}=\chi(\gamma),\qquad\gamma\in\Gamma_\infty.
\end{equation}
Let $ h\in\bbR_{>0} $ such that 
\begin{equation}\label{eq:t017}
	Z\left(\Met\right)\Gamma_\infty=Z\left(\Met\right)\left<n_h\right>.
\end{equation}

It is well-known that the classical Poincar\'e series
\[ \psi_{\Gamma,n,m,\chi}:=P_{\Gamma_\infty\backslash\Gamma,\chi}e^{2\pi in\frac\spacedcdot h},\qquad n\in\bbZ_{>0}, \]
converge absolutely and locally uniformly on $ \calH $ and define elements of $ S_m(\Gamma,\chi) $. 
The following lemma is a half-integral weight variant of \cite[Theorem 2-10]{MuicLFunk} that generalizes their construction. 

\begin{lemma}\label{lem:l010}
	Let $ (a_n)_{n\in\bbZ_{>0}}\subseteq\bbC $ such that $ \sum_{n=1}^\infty\abs{a_n}n^{1-\frac m2}<\infty $. Then, the double series
	\begin{equation}\label{eq:l005}
		S:=\sum_{\gamma\in\Gamma_\infty\backslash\Gamma}\sum_{n=1}^\infty\overline{\chi(\gamma)} a_ne^{2\pi in\frac\spacedcdot h}\big|_m\gamma
	\end{equation}
	converges absolutely and uniformly on compact sets in $ \calH $ and defines an element of $ S_m(\Gamma,\chi) $ that coincides with 
	\begin{equation}\label{eq:l006}
		P_{\Gamma_\infty\backslash\Gamma,\chi}\left(\sum_{n=1}^\infty a_ne^{2\pi in\frac\spacedcdot h}\right)=\sum_{n=1}^\infty a_n\psi_{\Gamma,n,m,\chi}.
	\end{equation}
	All the series in \eqref{eq:l006} converge absolutely and uniformly on compact sets in $ \calH $.
	Moreover, the series $ \sum_{n=1}^\infty a_n\psi_{\Gamma,n,m,\chi} $ converges to $ S $ in the topology of $ S_m(\Gamma,\chi) $.
\end{lemma}

\begin{proof}
	The terms of the series $ S $ are well-defined, i.e., we have
	\begin{equation}\label{eq:l050}
		e^{2\pi in\frac\spacedcdot h}\big|_m\gamma=\chi(\gamma)e^{2\pi in\frac\spacedcdot h},\qquad\gamma\in\Gamma_\infty,
	\end{equation}
	since by \eqref{eq:l051} the equality \eqref{eq:l050} is equivalent to $ h $-periodicity of the function $ e^{2\pi in\frac\spacedcdot h} $.
	
	The double series $ S $ and all the series in \eqref{eq:l006} converge absolutely and uniformly on compact sets in $ \calH $ by {\cite[Corollary 2.6.4]{miyake}} and the following estimate:
	\begin{align*}
		\int_{\Gamma\backslash\calH}&\sum_{\gamma\in\Gamma_\infty\backslash\Gamma}\sum_{n=1}^\infty\abs{a_n\left(e^{2\pi i n\frac\spacedcdot h}\big|_m\gamma\right)(z)\Im(z)^{\frac m2}}\,dv(z)\\
		&\overset{\eqref{eq:t011}}=\sum_{n=1}^\infty\abs{a_n}\int_{\Gamma\backslash\calH}\sum_{\gamma\in\Gamma_\infty\backslash\Gamma}\abs{e^{2\pi i n\frac{\gamma.z}h}\Im(\gamma.z)^{\frac m2}}\,dv(z)\\
		&\ =\ \sum_{n=1}^\infty\abs{a_n}\int_{\Gamma_\infty\backslash\calH}\abs{e^{2\pi i n\frac zh}\Im(z)^{\frac m2}}\,dv(z)\\
		&\ =\ \sum_{n=1}^\infty\abs{a_n}\int_0^h\int_0^\infty e^{-2\pi n\frac yh}y^{\frac m2-2}\,dy\,dx\\
		&\ =\ \sum_{n=1}^\infty\abs{a_n}n^{1-\frac m2}\frac{h^{\frac m2}}{(2\pi)^{\frac m2-1}}\Gamma\left(\frac m2-1\right)<\infty.
	\end{align*}
	In the last equality we use the standard notation for the gamma function: $ \Gamma(s):=\int_0^\infty t^{s-1}e^{-t}\,dt $, $ \Re(s)>0 $.
	Now it is clear that $ S=\eqref{eq:l006}$.
	
	Since $ S_m(\Gamma,\chi) $ is a finite-dimensional Hausdorff topological vector space, it is a closed subspace of $ C(\calH) $ with the topology of uniform convergence on compact sets. Thus, the fact that the series $ \sum_{n=1}^\infty a_n\psi_{\Gamma,n,m,\chi} $ converges to $ S $ uniformly on compact sets in $ \calH $ implies that $ S $ belongs to $ S_m(\Gamma,\chi) $ and that the series $ \sum_{n=1}^\infty a_n\psi_{\Gamma,n,m,\chi} $ converges to $ S $ in the topology of $ S_m(\Gamma,\chi) $.
\end{proof}

The following lemma (cf.~\cite[Lemma 5-1]{MuicLFunk}) will be applied in the proof of Lemma \ref{lem:f175} to prove the estimate \eqref{eq:l016} that will be used in the proof of Theorem \ref{thm:f121}.

\begin{lemma}\label{lem:t015}
	Let $ a\in\bbR_{>\frac12} $. Then,
	\begin{equation}\label{eq:f118}
		\Gamma(a)\int_{\bbR}\frac{dx}{\left(x^2+1\right)^a}=\sqrt\pi\,\Gamma\left(a-\frac12\right).
	\end{equation}
\end{lemma}

\begin{proof}
	We have
	\begin{align*}
		\Gamma(a)&\int_\bbR\frac{dx}{\left(x^2+1\right)^a}=\int_\bbR\int_0^\infty\left(\frac y{x^2+1}\right)^{a-1}e^{-y}\frac{dy}{x^2+1}\,dx\\
		&=\int_\bbR\int_0^\infty t^{a-1}e^{-\left(x^2+1\right)t}\,dt\,dx
		=\int_0^\infty t^{a-1}e^{-t}\,\left(\int_\bbR e^{-x^2t}\,dx\right)\,dt\\
		&=\int_0^\infty t^{a-\frac32}e^{-t}\,\left(\int_\bbR e^{-u^2}\,du\right)\,dt
		=\sqrt\pi\,\Gamma\left(a-\frac12\right).\qedhere
	\end{align*}
\end{proof}

By the half-integral weight version of \cite[Theorem 2.6.10]{miyake}, every $ f\in S_m(\Gamma,\chi) $ has the Fourier expansion
\[ f(z)=\sum_{n=1}^\infty a_n(f)e^{2\pi in\frac zh},\qquad z\in\calH, \]
where
\begin{equation}\label{eq:l009}
	a_n(f)=\frac{\varepsilon_\Gamma(4\pi n)^{m-1}}{h^m\Gamma(m-1)}\scal f{\psi_{\Gamma,n,m,\chi}}_{S_m(\Gamma,\chi)},\qquad n\in\bbZ_{>0}.
\end{equation}
The $ L $-function of $ f $ is the function $ L(\spacedcdot,f):\bbC_{\Re(s)>\frac m2+1}\to\bbC $,
\begin{equation}\label{eq:f174}
	L(s,f):=\sum_{n=1}^\infty\frac{a_n(f)}{n^s},\qquad \Re(s)>\frac m2+1.
\end{equation}
Since by a standard argument $ a_n(f)=O\left(n^{\frac m2}\right) $ (e.g., see \cite[Corollary 2.1.6]{miyake}), the series \eqref{eq:f174} converges absolutely and uniformly on compact sets in $ \bbC_{\Re(s)>\frac m2+1} $, hence $ L(\spacedcdot,f) $ is holomorphic on $ \bbC_{\Re(s)>\frac m2+1} $.

We define $ F:\calH\times\bbC\to\bbC $, 
\[ F(z,s):=\sum_{n=1}^\infty n^{s-1}e^{2\pi in\frac zh}. \]
The series on the right-hand side obviously converges absolutely and locally uniformly on $ \calH\times\bbC $, hence $ F $ is holomorphic on $ \calH\times\bbC $. 

Next, let $ \Re(s)>\frac m2+1 $. We have
\begin{equation}\label{eq:f115}
	L(s,f)\overset{\eqref{eq:l009}}=\frac{\varepsilon_\Gamma(4\pi)^{m-1}}{h^m\Gamma(m-1)}\sum_{n=1}^\infty\scal f{n^{m-1-\overline s}\psi_{\Gamma,n,m,\chi}}_{S_m(\Gamma,\chi)}.
\end{equation}
By Lemma \ref{lem:l010}, the series $ \sum_{n=1}^\infty n^{m-1-\overline s}\psi_{\Gamma,n,m,\chi} $ converges in $ S_m(\Gamma,\chi) $ and we have
\[ \sum_{n=1}^\infty n^{m-1-\overline s}\psi_{\Gamma,n,m,\chi}=P_{\Gamma_\infty\backslash\Gamma,\chi}(F\left(\spacedcdot,m-\overline s\right)), \]
where the latter series converges absolutely and uniformly on compact sets in $ \calH $.
Thus, it follows from \eqref{eq:f115} that 
\begin{equation}\label{eq:011}
	L(s,f)=\frac{\varepsilon_\Gamma(4\pi)^{m-1}}{h^m\Gamma(m-1)}\scal f{P_{\Gamma_\infty\backslash\Gamma,\chi}(F\left(\spacedcdot,m-\overline s\right))}_{S_m(\Gamma,\chi)},\quad\Re(s)>\frac m2+1.
\end{equation}
This motivates the following lemma.

\begin{lemma}\label{lem:f175}
	Let $ \Re(s)<\frac m2-1 $ or $ 1<\Re(s)<\frac m2 $. Then, the series
	\[ \Psi_{\Gamma,m,\chi,s}:=P_{\Gamma_\infty\backslash\Gamma,\chi}(F\left(\spacedcdot,s\right)) \]
	converges absolutely and uniformly on compact sets in $ \calH $ and defines an element of $ S_m(\Gamma,\chi) $.	
\end{lemma}

\begin{proof}
	In the case when $ \Re(s)<\frac m2-1 $, the claim follows from the above discussion. 
	
	Let us prove the claim in the case when $ 1<\Re(s)<\frac m2 $ by applying Lemma \ref{lem:t011} to the Poincar\'e series $ P_{\Gamma_\infty\backslash\Gamma,\chi}(F\left(\spacedcdot,s\right)) $. We prove the inequality
	\begin{equation}\label{eq:f116}
		\int_{\Gamma_\infty\backslash\calH}\abs{F(z,s)\Im(z)^{\frac m2}}\,dv(z)<\infty,\qquad 1<\Re(s)<\frac m2,
	\end{equation}
	the other conditions of Lemma \ref{lem:t011} being satisfied trivially.
	For every $ \varepsilon\in\bbR_{>0} $, we can write the left-hand side of \eqref{eq:f116} as the sum
	\[ \underbrace{\int_{[0,h]\times]0,\varepsilon[}\abs{F(z,s)\Im(z)^{\frac m2}}\,dv(z)}_{\textstyle =:I_1}+\underbrace{\int_{[0,h]\times[\varepsilon,\infty[}\abs{F(z,s)\Im(z)^{\frac m2}}\,dv(z)}_{\textstyle =:I_2}. \]
	First, we estimate $ I_2 $: we have
	\begin{align}
		I_2&=\int_{[0,h]\times[\varepsilon,\infty[}\abs{\sum_{n=1}^\infty n^{s-1}e^{2\pi in\frac zh}\Im(z)^{\frac m2}}\,dv(z)\nonumber\\
		\qquad&\leq\sum_{n=1}^\infty n^{\Re(s)-1}\int_0^h\int_\varepsilon^\infty e^{-2\pi n\frac yh}y^{\frac m2-2}\,dy\,dx\nonumber\\
		&\leq\left(\sum_{n=1}^\infty n^{\Re(s)-1}e^{-2\pi(n-1)\frac\varepsilon h}\right)h\int_\varepsilon^\infty e^{-2\pi\frac yh}y^{\frac m2-2}\,dy\nonumber\\
		&\leq\left(\sum_{n=1}^\infty n^{\Re(s)-1}e^{-2\pi(n-1)\frac\varepsilon h}\right)\frac{h^{\frac m2}}{(2\pi)^{\frac m2-1}}\Gamma\left(\frac m2-1\right),\label{eq:f173}
	\end{align}
	and the right-hand side is finite by d'Alembert's ratio test.
	To estimate $ I_1 $, we will use the Lipschitz summation formula (see \cite[Theorem 1]{knopp2001lipschitz}):
	\begin{equation}\label{eq:l011}
		\sum_{n=1}^\infty n^{s-1}e^{2\pi inz}=\Gamma(s)(2\pi)^{-s}e^{i\frac\pi2 s}\sum_{n\in\bbZ}\frac1{(z+n)^s},\qquad z\in\calH,\ \Re(s)>1,
	\end{equation}
	where
	\[ z^s:=e^{s\left(\log\abs z+i\arg(z)\right)},\qquad z\in\calH,\ s\in\bbC,\ \arg(z)\in\left]0,\pi\right[. \]
	We note that for all $ z\in\calH $ and $ s\in\bbC $,
	\begin{equation}\label{eq:l012}
		\abs{z^s}=\abs{z}^{\Re(s)}e^{-\Im(s)\arg(z)}\in\left[\abs z^{\Re(s)}\min\left\{e^{-\pi\Im(s)},1\right\},\abs z^{\Re(s)}\max\left\{e^{-\pi\Im(s)},1\right\}\right].
	\end{equation}
	Now, for $ 1<\Re(s)<\frac m2 $ we have
	\begin{align}
		&I_1\overset{\eqref{eq:l011}}=\int_{[0,h]\times]0,\varepsilon[}\abs{\Gamma(s)(2\pi)^{-s}e^{i\frac\pi2s}h^s\sum_{n\in\bbZ}\frac1{(z+nh)^s}\Im(z)^{\frac m2}}\,dv(z)\nonumber\\
		&\overset{\eqref{eq:l012}}\leq\Gamma(\Re(s))\left(\frac h{2\pi}\right)^{\Re(s)}e^{-\frac\pi2\Im(s)}\sum_{n\in\bbZ}\int_{[0,h]\times]0,\varepsilon[}\frac{\Im(z)^{\frac m2}}{\abs{z+nh}^{\Re(s)}}\,dv(z)\,\max\left\{e^{\pi\Im(s)},1\right\}\nonumber\\
		&\overset{\phantom{\eqref{eq:f118}}}=\Gamma(\Re(s))\left(\frac h{2\pi}\right)^{\Re(s)}e^{\frac\pi2\abs{\Im(s)}}\int_{\bbR\times]0,\varepsilon[}\frac{\Im(z)^{\frac m2}}{\abs z^{\Re(s)}}\,dv(z)\nonumber\\
		&\overset{\phantom{\eqref{eq:f118}}}=\Gamma(\Re(s))\left(\frac h{2\pi}\right)^{\Re(s)}e^{\frac\pi2\abs{\Im(s)}}\int_\bbR\int_0^\varepsilon\frac{y^{\frac m2-\Re(s)-2}}{\left(\left(\frac xy\right)^2+1\right)^{\frac{\Re(s)}2}}\,dy\,dx\nonumber\\
		&\overset{\phantom{\eqref{eq:f118}}}=\Gamma(\Re(s))\left(\frac h{2\pi}\right)^{\Re(s)}e^{\frac\pi2\abs{\Im(s)}}\int_\bbR\frac{dx}{\left(x^2+1\right)^{\frac{\Re(s)}2}}\cdot\int_0^\varepsilon y^{\frac m2-\Re(s)-1}\,dy\nonumber\\
		&\overset{\eqref{eq:f118}}=\Gamma(\Re(s))\left(\frac h{2\pi}\right)^{\Re(s)}e^{\frac\pi2\abs{\Im(s)}}\frac{\sqrt\pi\,\Gamma\left(\frac{\Re(s)-1}2\right)}{\Gamma\left(\frac{\Re(s)}2\right)}\cdot\frac{\varepsilon^{\frac m2-\Re(s)}}{\frac m2-\Re(s)}\label{eq:l016}<\infty.
	\end{align}
	This concludes the proof of \eqref{eq:f116}.
\end{proof}

The main result of this section is the following theorem.

\begin{theorem}\label{thm:f118}
	Let $ f\in S_m(\Gamma,\chi) $. Then, the function $ \ell(\spacedcdot,f):\bbC_{\Re(s)>\frac m2+1}\cup\bbC_{\frac m2<\Re(s)<m-1}\to\bbC $,
	\begin{equation}\label{eq:l014}
		\ell(s,f):=\frac{\varepsilon_\Gamma(4\pi)^{m-1}}{h^m\Gamma(m-1)}\scal f{\Psi_{\Gamma,m,\chi,m-\overline s}}_{S_m(\Gamma,\chi)},
	\end{equation}
	is well-defined and holomorphic. If $ m>4 $, then $ \ell(\spacedcdot, f) $ is an analytic continuation of $ L(\spacedcdot,f) $ to the half-plane $ \bbC_{\Re(s)>\frac m2} $.
\end{theorem}

\begin{proof}
	The function $ \ell(\spacedcdot,f) $ is well-defined by Lemma \ref{lem:f175}. Its restriction to the half-plane $ \bbC_{\Re(s)>\frac m2+1} $ is holomorphic since it coincides with $ L(\spacedcdot,f) $ by \eqref{eq:011}.
	
	Let us prove that $ \ell(\spacedcdot,f) $ is holomorphic on $ \bbC_{\frac m2<\Re(s)<m-1} $. We have
	\begin{align}
		\ell(s,f)&\overset{\eqref{eq:l014}}=\frac{(4\pi)^{m-1}}{h^m\Gamma(m-1)}\int_{\Gamma\backslash\calH}f(z)\sum_{\gamma\in\Gamma_\infty\backslash\Gamma}\chi(\gamma)\overline{\left(F\left(\spacedcdot,m-\overline s\right)\big|_m\gamma\right)(z)}\Im(z)^m\,dv(z)\nonumber\\
		&\underset{\eqref{eq:t011}}{\overset{\eqref{eq:t012}}=}\frac{(4\pi)^{m-1}}{h^m\Gamma(m-1)}\int_{\Gamma\backslash\calH}\sum_{\gamma\in\Gamma_\infty\backslash\Gamma}f(\gamma.z)\overline{F\left(\gamma.z,m-\overline s\right)}\Im(\gamma.z)^m\,dv(z)\nonumber\\
		&=\frac{(4\pi)^{m-1}}{h^m\Gamma(m-1)}\int_{\Gamma_\infty\backslash\calH}f(z)\overline{F\left(z,m-\overline s\right)}\Im(z)^m\,dv(z).\label{eq:t013}
	\end{align}
	Since the integrand in \eqref{eq:t013} is holomorphic in $ s $ for every fixed $ z\in\calH $, by \cite[Lemma 6.1.5]{miyake} it suffices to prove that 
	\begin{equation}\label{eq:t015}
		\int_{\Gamma_\infty\backslash\calH}\abs{f(z)\overline{F\left(z,m-\overline s\right)}\Im(z)^m}\,dv(z)
	\end{equation}
	is bounded, as a function of $ s $, on every compact subset of $ \bbC_{\frac m2<\Re(s)<m-1} $. This is true because \eqref{eq:t015} is dominated by
	\begin{equation}\label{eq:t016}
		\left(\sup_{z\in\calH}\abs{f(z)\Im(z)^{\frac m2}}\right)\cdot\int_{\Gamma_\infty\backslash\calH}\abs{F\left(z,m-\overline s\right)\Im(z)^{\frac m2}}\,dv(z),
	\end{equation}
	which is bounded, as a function of $ s $, on every compact subset of $ \bbC_{\frac m2<\Re(s)<m-1} $ since $ \sup_{z\in\calH}\abs{f(z)\Im(z)^{\frac m2}} $ is finite by a well-known argument (see \cite[Theorem 2.1.5]{miyake}), while the integral in \eqref{eq:t016} is dominated by the sum of \eqref{eq:f173} and \eqref{eq:l016} by the proof of Lemma \ref{lem:f175}.
	
	Let us prove the last claim of the theorem. If $ m>4 $, then $ m-1>\frac m2+1 $, hence the domain of $ \ell(\spacedcdot, f) $ is the half-plane $ \bbC_{\Re(s)>\frac m2} $. Since by \eqref{eq:011} $ \ell(\spacedcdot,f) $ coincides with $ L(\spacedcdot,f) $ on the half-plane $ \bbC_{\Re(s)>\frac m2+1} $, the claim follows.
\end{proof}

\section{Non-vanishing of $ L $-functions}\label{sec:f147}

Let $ m $, $ \Gamma $, $ \chi $, and $ h $ satisfy the assumptions of the first paragraph of Section \ref{sec:f146}. Let 
\[ N:=\inf\left\{\abs c:\begin{pmatrix}a&b\\c&d\end{pmatrix}\in P(\Gamma\setminus\Gamma_\infty)\right\}. \]	
It follows from \cite[Lemma 1.7.3]{miyake} that
\begin{equation}\label{eq:f121}
	hN\geq1.
\end{equation}
In particular, $ N>0 $.

In the proof of the following theorem, we will use the notion of the median, $ \mathrm M_{\Gamma(a,b)} $, of the gamma distribution with parameters $ a,b\in\bbR_{>0} $. It is defined as the unique $ \mathrm M_{\Gamma(a,b)}\in\bbR_{>0} $ such that
\[ \int_0^{\mathrm M_{\Gamma(a,b)}}x^{a-1}e^{-\frac xb}\,dx=\int_{\mathrm M_{\Gamma(a,b)}}^\infty x^{a-1}e^{-\frac xb}\,dx. \] 
By \cite[Theorem 1]{chen1986bounds}, we have
\begin{equation}\label{eq:f119}
	a-\frac13<\mathrm M_{\Gamma(a,1)}<a,\qquad a\in\bbR_{>0}.
\end{equation}

\begin{theorem}\label{thm:f121}
	Suppose that
	\begin{equation}\label{eq:l055}
		m\geq\frac{4\pi}{Nh}+\frac83. 
	\end{equation}
	Let $ 1<\Re(s)<\frac m2 $. If
	\begin{equation}\label{eq:l059}
		e^{\frac\pi2\abs{\Im(s)}}\Gamma\left(\frac{\Re(s)+1}2\right)\Gamma\left(\frac{\Re(s)-1}2\right)\frac{2^{\frac m2-1}}{\Gamma\left(\frac m2-1\right)}\frac{\left(\frac{\pi}{Nh}\right)^{\frac m2-\Re(s)}}{\frac m2-\Re(s)}\leq\pi,
	\end{equation}
	then $ \Psi_{\Gamma,m,\chi,s}\not\equiv0 $.
\end{theorem}

\begin{proof}
	We will prove the theorem by applying Theorem \ref{thm:f098} to the Poincar\'e series $ \Psi_{\Gamma,m,\chi,s}= P_{\Gamma_\infty\backslash\Gamma,\chi}\left(F\left(\spacedcdot,s\right)\right) $ with $ \babs\spacedcdot=\abs\spacedcdot $. Let us prove that the assumptions of Theorem \ref{thm:f098} are satisfied. The function $ F(\spacedcdot,s) $ satisfies the condition (f1) for $ \Lambda=\Gamma_\infty $ since it is $ h $-periodic and \eqref{eq:l051} holds, and by \eqref{eq:f116} it satisfies (f2). The condition \eqref{eq:f081} is satisfied by \eqref{eq:l051} and \eqref{eq:t014}. 
	
	Next, we define
	\[ S:=[0,h[\times\left]\frac1N,\infty\right[. \]
	Let us prove that $ S $ has the property (S1) of Theorem \ref{thm:f098}.
	Suppose that $ z\in S $ and $ \gamma=\left(\begin{pmatrix}a&b\\c&d\end{pmatrix},\eta_\gamma\right)\in\Gamma $ are such that $ \gamma.z\in S $. We need to prove that $ \gamma.z=z $. First we note that $ c=0 $, otherwise we would have
	\[ \frac1N<\Im(\gamma.z)\overset{\eqref{eq:t011}}=\frac{\Im(z)}{(c\,\Re(z)+d)^2+\left(c\,\Im(z)\right)^2}\leq\frac{\Im(z)}{\left(c\,\Im(z)\right)^2}=\frac1{c^2\Im(z)}<\frac1{N^2\frac1N}=\frac1N. \]
	Thus, $ \gamma\in\Gamma_\infty $, hence by \eqref{eq:t017} $ \gamma.z=z+kh $ for some $ k\in\bbZ $. Since $ \Re(z),\Re(\gamma.z)\in [0,h[ $, it follows that $ k=0 $, hence $ \gamma.z=z $.
	
	It remains to prove that $ S $ has the property (S2), i.e., writing $ (\Gamma_\infty.S)^c:=\calH\setminus\Gamma_\infty.S $, to prove the inequality
	\begin{equation}\label{eq:l053}
		\int_{\Gamma_\infty\backslash\Gamma_\infty.S}\abs{F(z,s)\Im(z)^{\frac m2}}\,dv(z)>\int_{\Gamma_\infty\backslash(\Gamma_\infty.S)^c}\abs{F(z,s)\Im(z)^{\frac m2}}\,dv(z).
	\end{equation}
	
	We have
	\begin{align*}
		&\int_{\Gamma_\infty\backslash\Gamma_\infty.S}\abs{F(z,s)\Im(z)^{\frac m2}}\,dv(z)\\
		&\overset{\textup{(S1)}}=\int_{S}\abs{\sum_{n=1}^\infty n^{s-1}e^{2\pi in\frac zh}\Im(z)^{\frac m2}}\,dv(z)\\
		&\overset{\phantom{\textup{(S1)}}}=\int_{\frac1N}^\infty\int_0^h\abs{\sum_{n=1}^\infty n^{s-1}e^{2\pi i(n-1)\frac{x+iy}h}}\,dx\,e^{-2\pi\frac yh}y^{\frac m2-2}\,dy\\
		&\overset{\phantom{\textup{(S1)}}}\geq\int_{\frac1N}^\infty\abs{\int_0^h\sum_{n=1}^\infty n^{s-1}e^{2\pi i(n-1)\frac{x+iy}h}\,dx}\,e^{-2\pi\frac yh}y^{\frac m2-2}\,dy\\
		&\overset{\phantom{\textup{(S1)}}}=\int_{\frac1N}^\infty\abs{\sum_{n=1}^\infty n^{s-1}\int_0^h e^{2\pi i(n-1)\frac{x+iy}h}\,dx}\,e^{-2\pi\frac yh}y^{\frac m2-2}\,dy\\
		&\overset{\phantom{\textup{(S1)}}}=h\left(\frac h{2\pi}\right)^{\frac m2-1}\int_{\frac{2\pi}{Nh}}^\infty e^{-y}y^{\frac m2-2}\,dy.
	\end{align*}
	Since
	\[ \frac{2\pi}{Nh}\overset{\eqref{eq:l055}}\leq\frac m2-\frac43\overset{\eqref{eq:f119}}<M_{\Gamma\left(\frac m2-1,1\right)}, \]
	it follows that
	\begin{align}
		\int_{\Gamma_\infty\backslash\Gamma_\infty.S}\abs{F(z,s)\Im(z)^{\frac m2}}\,dv(z)
		&>h\left(\frac h{2\pi}\right)^{\frac m2-1}\frac12\int_0^\infty e^{-y}y^{\frac m2-2}\,dy\nonumber\\
		&=\left(\frac h{2\pi}\right)^{\frac m2}\pi\,\Gamma\left(\frac m2-1\right).\label{eq:l058}
	\end{align}
	
	On the other hand, since $ [0,h[\times\left]0,\infty\right[ $ is a fundamental domain for the action of $ \Gamma_\infty $ on $ \calH $, we have
	\begin{align}
		\multispan2{$\displaystyle \int_{\Gamma_\infty\backslash(\Gamma_\infty.S)^c}\abs{F(z,s)\Im(z)^{\frac m2}}\,dv(z)=\int_{[0,h[\times\left]0,\frac1N\right]}\abs{F(z,s)\Im(z)^{\frac m2}}\,dv(z)$\hfil}\nonumber\\
		\quad&\overset{\eqref{eq:l016}}\leq e^{\frac\pi2\abs{\Im(s)}}\left(\frac h{2\pi}\right)^{\Re(s)}\frac{\sqrt\pi\,\Gamma(\Re(s))}{\Gamma\left(\frac{\Re(s)}2\right)}\,\frac{\Gamma\left(\frac{\Re(s)-1}2\right)\left(\frac1N\right)^{\frac m2-\Re(s)}}{\frac m2-\Re(s)}\nonumber\\
		&=\frac{e^{\frac\pi2\abs{\Im(s)}}}2\left(\frac{h}{\pi}\right)^{\Re(s)}\frac{\Gamma\left(\frac{\Re(s)+1}2\right)\Gamma\left(\frac{\Re(s)-1}2\right)\left(\frac1N\right)^{\frac m2-\Re(s)}}{\frac m2-\Re(s)},\label{eq:f120}
	\end{align}
	where the last equality is obtained by applying the Legendre duplication formula \cite[19.(2)]{rainville1960special}:
	\[ \frac{\sqrt\pi\,\Gamma(2z)}{2^{2z}\Gamma(z)}=\frac12\Gamma\left(z+\frac12\right),\qquad z\in\bbC\setminus\frac12\bbZ_{\leq0}. \]
	Inequalities \eqref{eq:l058} and \eqref{eq:f120} imply that \eqref{eq:l053} holds if we have
	\[ \left(\frac h{2\pi}\right)^{\frac m2}\pi\Gamma\left(\frac m2-1\right)
	\geq\frac{e^{\frac\pi2\abs{\Im(s)}}}2\left(\frac{h}{\pi}\right)^{\Re(s)}\frac{\Gamma\left(\frac{\Re(s)+1}2\right)\Gamma\left(\frac{\Re(s)-1}2\right)\left(\frac1N\right)^{\frac m2-\Re(s)}}{\frac m2-\Re(s)}, \]
	and this inequality is obviously equivalent to \eqref{eq:l059}.
\end{proof}

\begin{corollary}\label{cor:l054}
	Suppose that $ m\in\frac92+\bbZ_{\geq0} $. Let $ \frac m2<\Re(s)<m-1 $. Suppose that
	\[  \frac{Nh}\pi\geq\max\left\{\frac{4}{m-\frac83},\left(\frac{e^{\frac\pi2\abs{\Im(s)}}\Gamma\left(\frac{m-\Re(s)+1}2\right)\Gamma\left(\frac{m-\Re(s)-1}2\right)2^{\frac m2-1}}{\pi\Gamma\left(\frac m2-1\right)\left(\Re(s)-\frac m2\right)}\right)^{\frac1{\Re(s)-\frac m2}}\right\}. \]
	Then,
	\[ L\left(s,\Psi_{\Gamma,m,\chi,m-\overline s}\right)>0. \]
\end{corollary}

\begin{proof}
	By Theorem \ref{thm:f121} $ \Psi_{\Gamma,m,\chi,m-\overline s}\not\equiv0 $, hence 
	\[ 0<\scal{\Psi_{\Gamma,m,\chi,m-\overline s}}{\Psi_{\Gamma,m,\chi,m-\overline s}}_{S_m(\Gamma,\chi)}\overset{\eqref{eq:l014}}=\frac{h^m\Gamma(m-1)}{\varepsilon_\Gamma(4\pi)^{m-1}}\ell(s,\Psi_{\Gamma,m,\chi,m-\overline s}). \]
	Thus, $ \ell(s,\Psi_{\Gamma,m,\chi,m-\overline s})>0 $, i.e., by Theorem \ref{thm:f118}, $ L(s,\Psi_{\Gamma,m,\chi,m-\overline s})>0 $.
\end{proof}

\begin{corollary}\label{cor:l055}
	Let $ \varepsilon\in\bbR_{>\frac12} $, $ \nu\in\bbR_{>\varepsilon} $, and $ \eta\in\bbR_{>0} $. For $ m\in\frac92+\bbZ_{\geq0} $, we define
	\[ C_m:=\left[\frac m2+\varepsilon,\frac m2+\nu\right]\times[-\eta,\eta]\subseteq\bbC. \]
	There exists $ m_0\in \frac92+\bbZ_{\geq0} $ such that for all $ m\in m_0+\bbZ_{\geq0} $ and $ s\in C_m $, for every discrete subgroup $ \Gamma $ of finite covolume in $ \Met $ such that $ \infty $ is a cusp of $ P(\Gamma) $ and for every character $ \chi:\Gamma\to\bbC^\times $ of finite order that satisfies $ \eta_\gamma^{-2m}=\chi(\gamma) $ for all $ \gamma\in\Gamma_\infty $, we have
	\[ L\left(s,\Psi_{\Gamma,m,\chi,m-\overline s}\right)>0. \]
\end{corollary}

\begin{proof}
	If $ m>2\nu+2 $, then $ \frac m2+\nu<m-1 $, hence $ C_m\subseteq\bbC_{\frac m2<\Re(s)<m-1} $. Thus, by Corollary \ref{cor:l054} and by \eqref{eq:f121} it suffices to prove that there exists $ m_0\in\frac92+\bbZ_{\geq0} $ such that for all $ m\in m_0+\bbZ_{\geq0} $ and $ s\in C_m $, we have
	\begin{equation}\label{eq:l060}
		\frac1\pi\geq\frac4{m-\frac83}
	\end{equation}
	and
	\begin{equation}\label{eq:l061}
		\left(\frac1\pi\right)^{\Re(s)-\frac m2}\\
		\geq\frac{e^{\frac\pi2\abs{\Im(s)}}\Gamma\left(\frac{m-\Re(s)+1}2\right)\Gamma\left(\frac{m-\Re(s)-1}2\right)2^{\frac m2-1}}{\pi\Gamma\left(\frac m2-1\right)\left(\Re(s)-\frac m2\right)}.
	\end{equation}
	The inequality \eqref{eq:l060} obviously holds if $ m\geq\frac83+4\pi $. Next, if $ m\geq2\nu+10 $, then for every $ s\in C_m $, $ \frac{m-\Re(s)+1}2 $ and $ \frac{m-\Re(s)-1}2 $ belong to the interval $ \left[2,\infty\right[ $, on which the gamma function is non-decreasing, hence the right-hand side of \eqref{eq:l061} is bounded from above for all $ s\in C_m $ by
	\[ \frac{e^{\frac\pi2\eta}\Gamma\left(\frac m4+\frac{1-\varepsilon}2\right)\Gamma\left(\frac m4-\frac{1+\varepsilon}2\right)2^{\frac m2-1}}{\pi\Gamma\left(\frac m2-1\right)\varepsilon}. \]
	On the other hand, the left-hand side of \eqref{eq:l061} is, for all $ s\in C_m $, greater than or equal to $ \left(\frac1\pi\right)^\nu $, hence \eqref{eq:l061} holds for all $ s\in C_m $ if
	\[ \left(\frac1\pi\right)^\nu\geq\frac{e^{\frac\pi2\eta}\Gamma\left(\frac m4+\frac{1-\varepsilon}2\right)\Gamma\left(\frac m4-\frac{1+\varepsilon}2\right)2^{\frac m2-1}}{\pi\Gamma\left(\frac m2-1\right)\varepsilon}, \]
	i.e., if we have
	\[ \frac\varepsilon2e^{-\frac\pi2\eta}\pi^{\frac12-\nu}\geq \Gamma\left(\frac m4+\frac{1-\varepsilon}2\right)\Gamma\left(\frac m4-\frac{1+\varepsilon}2\right)\frac{2^{\frac m2-2}}{\sqrt\pi\,\Gamma\left(\frac m2-1\right)}=:R(m). \]
	Thus, to finish the proof of the corollary, it suffices to prove that $ \lim_{m\to\infty}R(m)=0 $. By applying the Legendre duplication formula \cite[19.(2)]{rainville1960special}
	\[ \frac{2^{2z-1}}{\sqrt\pi\,\Gamma(2z)}=\frac1{\Gamma(z)\Gamma\left(z+\frac12\right)},\qquad z\in\bbC\setminus\frac12\bbZ_{\leq0}, \]
	we see that
	\[ R(m)=\frac{\Gamma\left(\frac m4+\frac{1-\varepsilon}2\right)\Gamma\left(\frac m4-\frac{1+\varepsilon}2\right)}{\Gamma\left(\frac m4-\frac12\right)\Gamma\left(\frac m4\right)}. \]
	Let us write $ m=4n+l $ with $ n\in\bbZ_{>0} $ and $ l\in\left\{\frac12,\frac32,\frac52,\frac72\right\} $. We note that $ R\left(4n+l\right) $ equals
	\begin{equation}\label{eq:l062}
		{\frac{\Gamma\left(n+\frac l4+\frac{1-\varepsilon}2\right)}{\Gamma(n)n^{\frac l4+\frac{1-\varepsilon}2}}\cdot\frac{\Gamma\left(n+\frac l4-\frac{1+\varepsilon}2\right)}{\Gamma(n)n^{\frac l4-\frac{1+\varepsilon}2}}}\left({\frac{\Gamma\left(n+\frac l4-\frac12\right)}{\Gamma(n)n^{\frac l4-\frac12}}\cdot\frac{\Gamma\left(n+\frac l4\right)}{\Gamma(n)n^{\frac l4}}}\right)^{-1}n^{\frac12-\varepsilon}.
	\end{equation}
	For every $ l\in\left\{\frac12,\frac32,\frac52,\frac72\right\} $, the right-hand side of \eqref{eq:l062} converges to $ 0 $ as $ n\to\infty $ since by \cite[18.~Lemma 7]{rainville1960special} we have
	\[ \lim_{n\to\infty}\frac{\Gamma(n+s)}{(n-1)!\,n^s}=1,\qquad s\in\bbC. \]
	Thus, $ \lim_{m\to\infty}R(m)=0 $. This proves the corollary.
\end{proof}

\end{document}